\documentclass{amsart}

\usepackage[english]{babel}
\usepackage[utf8]{inputenc}
\usepackage[T1]{fontenc}

\newtheorem{thm}{Theorem}[section]

\newtheorem{lem}[thm]{Lemma}

\theoremstyle{definition}

\DeclareMathOperator{\td}{td}

\title{Smooth versus symplectic circle actions}
\author{Łukasz Bąk}
\address{Jagiellonian University, Kraków}
\email{lukasz.bak@im.uj.edu.pl}

\begin{document}

\begin{abstract}
We construct a $6$-manifold $M$ which admits a smooth circle action and a symplectic form $\omega$ such that if $\omega'$ is another symplectic form on $M$ equivalent to $\omega$, then $(M,\omega')$ does not admit a symplectic circle action. 
\end{abstract}

\maketitle

\section{Introduction}\label{intro}

In this short note we are concerned with a question whether a manifold admitting a symplectic form and a smooth circle action has to admit a structure of a symplectic $\mathbb{S}^{1}$-manifold.

For a symplectic surface the answer is clearly positive. In case of symplectic $4$-manifolds the answer is positive if the action has a fixed point \cite{Bal04} and unknown in general. Little research has been conducted in dimensions $6$ and above.

The main result of this paper is the following theorem.

\begin{thm}\label{main}
There exists a symplectic $6$-manifold $(M,\omega)$ such that the underlying smooth manifold $M$ admits a non-trivial smooth circle action, but for any symplectic form $\omega'$ equivalent to $\omega$ the symplectic manifold $(M,\omega')$ admits no non-trivial symplectic circle actions.
\end{thm}
We say that two symplectic forms $\omega$ and $\omega'$ on a manifold $M$ are equivalent if there exists a sequence $(\omega,\omega_{1},\ldots,\omega_{k},\omega')$ of symplectic forms on $M$ such that each two consecutive forms in this sequence are either deformation equivalent or one of them is a pull-back of the other by a self-diffeomorphism of $M$.

Unfortunately, Theorem \ref{main} does not answer the question stated above. It is known that manifolds in general may admit non-equivalent symplectic forms (cf.~\cite{Smi00} for dimension $4$ and \cite{Rua94} for dimension $6$).

The proof of Theorem \ref{main} above is divided into two steps. First, in Section \ref{sectodd} below, we prove an inequality for Todd genus of arbitrary Hamiltonian $6$-manifolds. Then, in Section \ref{seccon}, we construct a symplectic $6$-manifold with a non-trivial smooth circle action, which violates this inequality and such that any symplectic circle action on this manifold has to be Hamiltonian. Since Todd genus of a symplectic manifold is invariant with respect to equivalence of symplectic forms, these two results combined constitute a proof of Theorem \ref{main}.

\section{Todd genus of Hamiltonian $6$-manifolds}\label{sectodd}

In this section we will prove the following lemma.

\begin{lem}\label{sec2main}
If $(M,\omega)$ is a Hamiltonian $6$-manifold, then
\[\td(M,\omega)\geq{1-\frac{b_{1}(M)}{2}},\]
where $b_{k}$ denotes the $k$th Betti number of $M$.
\end{lem}

By $\td(M,\omega)$ we denote the Todd genus of $(M,\omega)$, defined as the Todd genus of an almost complex manifold $(M,J)$, where $J$ is any almost complex structure on $M$ compatible with $\omega$. Since the space of all such almost complex structures on $M$ is contractible, this is in fact well-defined.

Fix any symplectic manifold $(M,\omega)$. For now we do not assume that $M$ is of dimension $6$. Fix a non-trivial Hamiltonian circle action on $M$ with Hamiltonian function $\phi$. It is well known that the critical manifolds of $\phi$, which correspond to the connected components of the fixed point set of the action, are symplectic submanifolds of $M$ and that exactly one of these submanifolds, which we will denote by $M_{\phi}$, corresponds to the minimum of the Hamiltonian. By studying $\phi$ as a Morse-Bott function, Li \cite{Li03} concluded that
\[\pi_{1}(M_{\phi})\cong\pi_{1}(M).\]
At the other hand, recall that the Todd genus is an evaluation of a Hirzebruch's $\chi_{y}$-genus at $y=0$. Combining this with the localization formula for $\chi_{y}$-genus of $\mathbb{S}^{1}$-manifolds \cite{BerHirJun92} it is not hard to obtain
\[\td(M_{\phi},\omega_{M_{\phi}})=\td(M,\omega).\]

\begin{proof}[Proof of Lemma \ref{sec2main}]
Let now assume that $\dim{M}=6$. As mentioned earlier, $M_{\phi}$ is a connected symplectic submanifold of $M$, so if the action is not trivial, then $\dim{M_{\phi}}\in\{0,2,4\}$. If $\dim{M_{\phi}}=0$, then $M_{\phi}$ is a point and
\[\td(M,\omega)=\td(\bullet)=1\geq{1-\frac{b_{1}(M)}{2}}.\]
If $\dim{M_{\phi}}=2$, then $M_{\phi}$ is an oriented surface $F_{g}$ of genus $g$ for some non-negative integer $g$ and $\pi_{1}(M)\cong\pi_{1}(F_{g})\cong\ast^{g}\mathbb{Z}^{2}$. In particular $b_{1}(M)=2g$. Now, the Todd genus of a surface does not depend on a choice of a symplectic (or almost complex) structure and equals $\td(F_{g})=1-g$. Thus
\[\td(M,\omega)=\td(F_{g})=1-g=1-\frac{b_{1}(M)}{2}.\]
Finally, if $\dim(M_{\phi})=4$ then $M_{\phi}$ is a symplectic $4$-manifold with $b_{1}(M_{\phi})=b_{1}(M)$. Similarly to the case of surfaces, the Todd genus of a $4$-manifold does not depend on the choice of an almost complex structure and equals $\td(M_{\phi})=\frac{1-b_{1}(M_{\phi})+b_{2}^{+}(M_{\phi})}{2}$, where $b_{2}^{+}(M_{\phi})$ is the dimension of the self-dual part of $H^{2}(M_{\phi};\mathbb{R})$. But $(M_{\phi},\omega|_{M_{\phi}})$ is symplectic, so $[\omega|_{M_{\phi}}]\in{H^{2}(M_{\phi};\mathbb{R})}$ is self-dual and $b_{2}^{+}(M_{\phi})\geq{1}$. In particular
\[\td(M,\omega)=\td(M_{\phi})\geq{1-\frac{b_{1}(M_{\phi})}{2}}=1-\frac{b_{1}(M)}{2}.\]
\end{proof}

\section{Construction of an example}\label{seccon}

Let $K$ denote the $K3$ surface. It is a symplectic manifold with Euler characteristic $\chi(K)=24$, signature $\sigma(K)=-16$, Todd genus $\td(K)=2$ and even intersection form. If we take a one point blow up $K\sharp\overline{\mathbb{C}P^{2}}$ of $K$, then we obtain a new symplectic manifold, now with $\chi(K\sharp\overline{\mathbb{C}P^{2}})=25$, $\sigma(K\sharp\overline{\mathbb{C}P^{2}})=-17$, $\td(K\sharp\overline{\mathbb{C}P^{2}})=2$ and odd intersection form. By Serre's classification of indefinite unimodular bilinear symmetric forms we see, that the intersection form of $K\sharp\overline{\mathbb{C}P^{2}}$ coincides with that of $3\mathbb{C}P^{2}\sharp20\overline{\mathbb{C}P^{2}}$, so they are homotopy equivalent. Fix an integer $g>1$. If $F_{g}$ is an oriented surface of genus $g$, then
\[M:=(K\sharp\overline{\mathbb{C}P^{2}})\times{F_{g}}\cong(3\mathbb{C}P^{2}\sharp20\overline{\mathbb{C}P^{2}})\times{F_{g}}.\]

\begin{lem}\label{Mham}
Let $\omega$ be a non-trivial symplectic form on $M$. Every symplectic circle action on $(M,\omega)$ is Hamiltonian.
\end{lem}

\begin{proof}
Since $\chi(M)=50(1-g)\neq{0}$, every circle action on $M$ has fixed points. Recall, that a symplectic $2n$-manifold $(X,\omega_{X})$ is said to have the weak Lefschetz property (WLP) if $\wedge[\omega_{X}]^{n-1}:H^{1}(X;\mathbb{R})\to{H^{2n-1}(X;\mathbb{R})}$ is an isomorphism. If a~symplectic manifold $(X,\omega_{X})$ has WLP, then a symplectic circle action on $(X,\omega_{X})$ is Hamiltonian if and only if it has fixed points (see \cite[Theorem 5.5]{McDSal98}). So it remains to see that $(M,\omega)$ has WLP.

Using K\"{u}nneth's formula we obtain decompositions
\begin{align*}
H^{1}(M;\mathbb{R})&\cong{H^{0}(K\sharp\overline{\mathbb{C}P^{2}};\mathbb{R})}\otimes{H^{1}(F_{g};\mathbb{R})},\\
H^{2}(M;\mathbb{R})&\cong{H^{0}(K\sharp\overline{\mathbb{C}P^{2}};\mathbb{R})}\otimes{H^{2}(F_{g};\mathbb{R})}\oplus{H^{2}(K\sharp\overline{\mathbb{C}P^{2}};\mathbb{R})}\otimes{H^{0}(F_{g};\mathbb{R})},\\
H^{5}(M;\mathbb{R})&\cong{H^{4}(K\sharp\overline{\mathbb{C}P^{2}};\mathbb{R})}\otimes{H^{1}(F_{g};\mathbb{R})}.
\end{align*}
In particular, $[\omega]$ decomposes as $[\omega]=1\otimes{b}+a\otimes{1}$ for some $a\in{H^{2}(K\sharp\overline{\mathbb{C}P^{2}};\mathbb{R})}$ and $b\in{H^{2}(F_{g};\mathbb{R})}$. Take any class $\gamma=1\otimes{c}\in{H^{1}(M;\mathbb{R})}$, where $c\in{H^{1}(F_{g};\mathbb{R})}$. We have
\[\gamma\wedge[\omega]^{2}=a^{2}\otimes{c}.\]
$\wedge[\omega]^{2}:H^{1}(M;\mathbb{R})\to{H^{5}(M;\mathbb{R})}$ is an isomorphism precisely when $a^{2}\neq{0}$. But $0\neq[\omega]^{3}=a^{2}\otimes{b}$.
\end{proof}

We can easily endow $3\mathbb{C}P^{2}\sharp20\overline{\mathbb{C}P^{2}}$ with a non-trivial smooth circle action by taking appropriate linear actions on each component and using equivariant connected sum. Product action on $M$ is non-trivial. Hence, we have shown the following.

\begin{lem}\label{Msymm}
$M$ admits a non-trivial smooth circle action.
\end{lem}

Finally, let $\omega_{K\sharp\overline{\mathbb{C}P^{2}}}$ be a symplectic form on $K\sharp\overline{\mathbb{C}P^{2}}$ and let $\omega_{F_{g}}$ be a symplectic form on $F_{g}$. Denote the product form on $M$ by $\omega$. Clearly, it is symplectic. Moreover, $\td(M,\omega)=\td(K\sharp\overline{\mathbb{C}P^{2}})\td(F_{g})=2(1-g)$. This gives us the following.

\begin{lem}\label{Msymp}
$M$ admits a symplectic form $\omega$ such that $\td(M,\omega)=2(1-g)$.
\end{lem}

Lemmas \ref{sec2main}, \ref{Mham}, \ref{Msymm} and \ref{Msymp} complete the proof of Theorem \ref{main}, as outlined in Section \ref{intro}.

\bibliography{smvssm}
\bibliographystyle{plain}

\end{document}